\documentclass[a4paper,12pt]{amsart}
\usepackage{amsmath}
\usepackage{amssymb}
\usepackage{enumerate}
\usepackage{amscd}
\usepackage{graphics}
\usepackage{latexsym}
\usepackage{verbatim} 
\usepackage{url}
\usepackage{stmaryrd}

\theoremstyle{plain}
\newtheorem{thm}{{\bf Theorem}}[section]
\newtheorem{cor}[thm]{{\bf  Corollary}}
\newtheorem{prop}[thm]{{\bf Proposition}}

\newtheorem{claim}[thm]{{\bf Claim}}

\theoremstyle{definition}

\newtheorem{question}[thm]{{\bf Question}}
\newtheorem{remark}[thm]{{\bf Remark}}

\newcommand{\cf}{\mathord{\mathrm{cf}}}

\newcommand{\dom}{\mathord{\mathrm{dom}}}
\newcommand{\size}[1]{\left\vert {#1} \right\vert}

\newcommand{\seq}[1]{\langle {#1} \rangle}

\newcommand{\ka}{\kappa}

\newcommand{\om}{\omega}

\newcommand{\bbr}[1]{\llbracket {#1} \rrbracket}

\newcommand{\bbR}{\mathbb{R}}

\newcommand{\calB}{\mathcal{B}}

\newcommand{\calU}{\mathcal{U}}
\newcommand{\calV}{\mathcal{V}}

\title[]{Products of Lindel\"of spaces with points $G_\delta$}
\author[T. Usuba]{Toshimichi Usuba}
\address[T. Usuba]
{Faculty of Science and Engineering,
Waseda University, 
Okubo 3-4-1, Shinjyuku, Tokyo, 169-8555 Japan}
\email{usuba@waseda.jp}

\begin{document}
\maketitle
\begin{abstract}
We show that if CH holds and either (i) there exists an $\om_1$-Kurepa tree, or
(ii) $\square(\om_2)$ holds, then
there are regular $T_1$ Lindel\"of spaces $X_0$ and $X_1$ with points $G_\delta$
such that the extent of $X_0 \times X_1$ is strictly greater than $2^\om$.
\end{abstract}

\section{Introduction}

While every product of compact spaces is compact,
the product of two Lindel\"of spaces need not to be Lindel\"of;
The Sorgenfrey line is a typical example.
The square of two Sorgengrey lines has the Lindel\"of degree $2^\om$,
where the \emph{Lindel\"of degree} of the space $X$, $L(X)$, is 
the minimal cardinal $\ka$ such that every open cover of $X$ has a subcover of size $\le \ka$.
This fact lead us to the following natural question.
\begin{question}
Are there two Lindel\"of spaces whose product has the Lindel\"of degree $>2^\om$?
\end{question}
Some consistent examples are known.
Shelah \cite{Shelah}
constructed a model of ZFC in which 
there are two regular $T_1$ Lindel\"of spaces  with points $G_\delta$
whose product has the extent $(2^\om)^+=\om_2$,
where the \emph{extent} of $X$, $e(X)$, is $\sup\{\size{C} \mid C \subseteq X$ is closed discrete$\}$.
It is clear that $L(X) \ge e(X)$.
Gorelic \cite{Gorelic} refined and simplified Shelah's method
and got a model in which
there are two regular $T_1$ Lindel\"of spaces with points $G_\delta$
whose product has the extent $2^{\om_1}$ and $2^{\om_1}$ is arbitrary large.
The extent of the product of their spaces is bounded by $2^{\om_1}$,
and Usuba \cite{Usuba} proved that it is consistent that 
the extent of the product of two regular $T_1$ Lindel\"of spaces can be arbitrary large up to the
least measurable cardinal.
However it is still open if the existence of such Lindel\"of spaces is provable from ZFC.

In this paper,  
we  give new construction of such Lindel\"of spaces under some combinatorial principles.
\begin{thm}\label{thm1}
Suppose CH.
If there exists an $\om_1$-Kurepa tree, or Todor\v cevi\'c's square principle $\square(\om_2)$ holds,
then
there are regular $T_1$ Lindel\"of spaces $X_0, X_1$ with points $G_\delta$
such that $e(X_0 \times X_1)>2^\om$.
\end{thm}
An \emph{$\om_1$-Kurepa tree} is an $\om_1$-tree having strictly more than
$\om_1$ cofinal branches.
We say that $\square(\om_2)$ \emph{holds} if there exists a sequence $\seq{c_\alpha \mid \alpha<\om_2}$
such that for each $\alpha<\om_2$, $c_\alpha$ is a club in $\alpha$,
$c_\beta =c_\alpha \cap \beta$ for every $\beta$ from the limit points of $c_\alpha$,
and there is no club $D$ in $\om_2$ such that
$D \cap \alpha=c_\alpha$ for every $\alpha$ from the limit points of $D$.

This theorem has some interesting consequences.
It is known that the following hold under $V=L$:
\begin{enumerate}
\item CH holds (G\"odel).
\item There exists an $\om_1$-Kurepa tree (Solovay, e.g., see Theorem 27.8 in Jech \cite{Jech}).
\end{enumerate}
Hence we have alternative proof of the following result by Shelah \cite{Shelah}:
\begin{cor}[Shelah \cite{Shelah}]\label{1.3+}
Suppose $V=L$.
Then there are regular $T_1$ Lindel\"of spaces $X_0, X_1$ with points $G_\delta$
such that $e(X_0 \times X_1)>2^\om$.
\end{cor}

It is also known that if $\square(\om_2)$ fails then
 $\om_2$ is weakly compact in $L$ (Todor\v cevi\'c, (1.10) in Todor\v cevi\'c \cite{T}).
\begin{cor}
Suppose CH.
If $e(X_0 \times X_1) \le 2^\om$ for every 
regular $T_1$ Lindel\"of spaces $X_0, X_1$ with points $G_\delta$,
then $\om_2$ is weakly compact in the constructible universe $L$.
\end{cor}
This shows that the non-existence of such Lindel\"of spaces would have a large cardinal strength
(if it is consistent).

A very rough sketch of our construction is as follows.
For a certain Hausdorff Lindel\"of space,
we modify open neighborhoods of each points of the space and  
construct finer Lindel\"of spaces $X_0$ and $X_1$ such that
for each $x \in X$, there are open sets $O_0 \subseteq  X_0$ and $O_1 \subseteq X_1$ with
$O_0 \cap O_1=\{x\}$. Clearly the diagonal of $X_0 \times X_1$ is a large closed discrete subset of
$X_0 \times X_1$. Basic idea of our construction come from Usuba \cite{Usuba2}.

\subsection*{Acknowledgements}
The author would like to thank the referee for many useful comments and suggestions.
This research was suppoted by JSPS KAKENHI Grant Nos. 18K03403 and 18K03404.


\section{Modifying points with character $\om_1$}
\begin{prop}\label{prop1}
Let $X$ be a Hausdorff Lindel\"of space of size $>2^\om$,
and $X_0$, $X_1$ be regular $T_1$ Lindel\"of spaces of character $\le \om_1$
such that:
\begin{enumerate}
\item $X_0$ and $X_1$ have the same underlying sets to $X$ and  topologies of $X_0$ and $X_1$ are finer than $X$.
\item For every $x \in X$, $\chi(x, X_0)=\chi(x, X_1)$.
\item For $x \in X$, if $\chi(x, X_0)=\chi(x, X_1)=\om_1$ then 
there exists a sequence $\seq{O_\alpha^x:\alpha<\om_1}$ with the following properties:
\begin{enumerate}
\item $O_\alpha^x$ is clopen in $X$.
\item $O_{\alpha}^x \supseteq O_{\alpha+1}^x$.
\item $O_\alpha^x=\bigcap_{\beta<\alpha} O_\beta^x$ if $\alpha$ is limit.
\item $\bigcap_{\alpha<\om_1} O_\alpha^x=\{x\}$.
\end{enumerate}
\item For $x \in X$, if $\chi(x, X_0)=\chi(x, X_1)=\om$ then 
there are open sets $O_0 \subseteq X_0$ and $O_1 \subseteq X_1$
respectively with $O_0 \cap O_1=\{x\}$.
\end{enumerate}
Then there are regular $T_1$ Lindel\"of spaces $Y_0$ and $Y_1$ with points $G_\delta$
such that $e(Y_0 \times Y_1)=\size{X}>2^\om$.
\end{prop}
\begin{proof}
First, fix an injection $\sigma: \om_1 \to \bbR$ where $\bbR$ is the real line.
Let $X'=\{x \in X_0 \mid \chi(x, X_0)=\om_1\}=\{x \in X_1 \mid \chi(x, X_1)=\om_1\}$.
For a set $A \subseteq X$, let $\bbr{A}=\bigcup \{\{x\} \times \bbR \mid x \in A \cap X' \} \cup (A \setminus X')$.

For $x \in X'$, $\alpha<\om_1$, and a set $W \subseteq \bbR$,
let $O(x, \alpha, W)=\bigcup \{\bbr{O^x_\beta \setminus O^x_{\beta+1}} \mid \beta \ge \alpha,
\sigma(\beta) \in W\} \cup (\{x\}\times W)$.

For  constructing $Y_0$,
let $S$ be the Sorgenfrey line,
that is, the underlying set of $S$ is the real line $\bbR$,
and the topology is generated by the family $\{[r,s) \mid r,s \in \bbR\}$ as an open base.
It is known that $S$ is a first countable regular $T_1$ Lindel\"of space.

We define $Y_0$ in the following manner. The underlying set of $Y_0$ is $\bbr{X}$.
The topology of $Y_0$ is generated by the family
$\{ \bbr{O} \mid O \subseteq X_0$ is open$\} \cup \{O(x, \alpha, W) \mid
x \in X', \alpha<\om_1, W \subseteq S$ is open$\}$
as an open base.
We know that $Y_0$ is a regular $T_1$ Lindel\"of space with points $G_\delta$
(see Proposition 1.2 in \cite{Usuba}).

For $Y_1$, let $S^*$ be the space $\bbR$ equipped with the reverse Sorgenfrey topology,
that is, the topology generated by the family $\{(r,s] \mid r,s \in \bbR\}$ as an open base.
As with $S$, $S^*$ is a first countable regular $T_1$ Lindel\"of space.
Then we define $Y_1$ by the same way to $Y_0$ but replacing $X_0$ by $X_1$ and $S$ by $S^*$.
Again, $Y_1$ is a regular $T_1$ Lindel\"of space with points $G_\delta$.

To show that $e(Y_0 \times Y_1)=\size{X}>2^\om$,
let $\Delta=\{\seq{x,x} \mid x \in X \setminus X'\} \cup
\{\seq{ \seq{x,r},\seq{x,r}} \mid x \in X', r \in \bbR\}$.
We see that $\Delta$ is closed and discrete.

For the closeness of $\Delta$, take $p \in (Y_0\times Y_1) \setminus \Delta$.

Case 1: $p=\seq{x,y}$ for some $x,y \in X \setminus X'$.
Since $X$ is Hausdorff, there are disjoint open sets $O_0, O_1 \subseteq X$ with
$x \in O_0$ and $y \in O_1$.
Since $X_0$ and $X_1$ are finer than $X$, $O_0$ and $O_1$ are open in
$X_0$ and $X_1$ respectively.
Then $\bbr{O_0} \subseteq Y_0$ is open with $x \in \bbr{O_0}$,
$\bbr{O_1} \subseteq Y_1$ is open with $y \in \bbr{O_1}$,
and $\bbr{O_0} \cap \bbr{O_1}=\emptyset$.
Hence $\seq{x,y} \in \bbr{O_0} \times \bbr{O_1}$ and $\Delta \cap (\bbr{O_0} \times \bbr{O_1})=\emptyset$.

Case 2: $p=\seq{x, \seq{y,r}}$ for some $x \in X \setminus X'$, $y \in X'$, and $r\in \bbR$.
Again, take open sets $O_0, O_1 \subseteq X$
such that $x \in O_0$, $y \in O_1$, and $O_0 \cap O_1=\emptyset$.
Then $x \in \bbr{O_0}$, $\seq{y,r} \in \bbr{O_1}$, and
$\bbr{O_0} \cap \bbr{O_1}=\emptyset$.
So $p \in \bbr{O_0} \times \bbr{O_1}$ and $\Delta \cap (\bbr{O_0} \times \bbr{O_1})=\emptyset$.

Case 3: $p=\seq{\seq{x,r}, y}$ for some $x \in X'$, $y \in X \setminus X'$, and $r\in \bbR$.
Similar to Case 2.

Case 4: $p=\seq{\seq{x,r}, \seq{y,s}}$ for some $x, y \in X'$ and $r, s \in \bbR$.
If $x \neq y$, we can take open sets $O_0, O_1 \subseteq X$
with $x \in O_0$, $y \in O_1$, and $O_0 \cap O_1=\emptyset$.
Then $\bbr{O_0} \times \bbr{O_1}$ is a required set.
If $x=y$ and $r \neq s$,
take open sets $W_0, W_1 \subseteq \bbR$ with $r \in W_0$, $s \in W_1$, and $W_0 \cap W_1=\emptyset$.
Now $\seq{x,r} \in O(x, 0, W_0)$, $\seq{y,s} \in O(y, 0, W_1)$, and
$O(x, 0, W_0) \cap O(y,0,W_1)=\emptyset$.
Hence $p \in O(x, 0, W_0) \times O(y,0,W_1)$
and $\Delta \cap (O(x, 0, W_0) \times O(y,0,W_1))=\emptyset$.

Next we see that $\Delta$ is discrete.
For $x \in X \setminus X'$, by the assumption,
there are open sets $O_0 \subseteq X_0$ and $O_1 \subseteq X_1$ respectively with
$O_0 \cap O_1=\{x\}$.
Then it is clear that $\bbr{O_0} \cap \bbr{O_1}=\{x\}$,
hence $\Delta \cap (\bbr{O_0} \times \bbr{O_1})=\{x\}$.
For $x \in X'$ and $r \in \bbR$,
consider open sets $W_0=[r, r+1)$ in $S$ and $W_1=(r-1, r]$ in $S^*$.
Trivially $W_0 \cap W_1=\{r\}$.
Then, by the definitions of $O(x, 0, W_0) \subseteq Y_0$ and
$O(x, 0, W_1) \subseteq Y_1$, we have
$O(x, 0, W_0) \cap O(x,0, W_1)=\{\seq{x, r}\}$.
Thus $\Delta \cap (O(x,0,W_0) \times O(x, 0, W_1))=\{\seq{x,r}\}$,
as required.
\end{proof}
A space $X$ is said to be a \emph{$P$-space} if
every $G_\delta$ subset of $X$ is open.
If $X$ is a regular $T_1$ Lindel\"of $P$-space of character $\le \om_1$,
then  every point $x \in X$ with $\chi(x, X)=\om$ is isolated in $X$. 
Hence $X=X_0=X_1$ satisfy the assumptions of the previous proposition.
\begin{cor}
If there exists a regular $T_1$ Lindel\"of $P$-space of character $\le \om_1$ and size $>2^\om$,
then there are regular $T_1$ Lindel\"of spaces $Y_0$, $Y_1$ with points $G_\delta$
such that $e(Y_0 \times Y_1)>2^\om$.
\end{cor}
It is known that such a $P$-space exists under $V=L$
(Juh\'asz-Weiss \cite{JW}),
so this fact yields one more another proof of Corollary \ref{1.3+}.
\section{Modifying points with character $\om$}
For our convenience, we fix some notations and definitions.
For an ordinal $\alpha$,
let $2^\alpha$ be the set of all functions from $\alpha$ to $2$,
and $2^{<\alpha}$ ($2^{\le \alpha}$, respectively) be
$\bigcup_{\beta<\alpha} 2^\beta$ ($\bigcup_{\beta \le \alpha} 2^{\beta}$, respectively).
We say that $T$ is a \emph{tree} if $T$ is a subset of $2^{<\alpha}$ for some ordinal $\alpha$ such that
$T$ is downward closed, that is, for every $s \in T$ and $t \in 2^{<\alpha}$,
if $t \subseteq s$ then $t \in T$.
For $s, t \in T$, define $s \le t \iff s \subseteq t$,
and $s <t \iff s \subsetneq t$.
A \emph{branch} of a tree $T$ is a maximal chain of $T$.
If $B$ is a branch, then $\bigcup B$ is a function with $\bigcup B \in 2^{\le \alpha}$
and $B=\{\bigcup B\restriction \beta \mid \beta<\dom(\bigcup B)\}$.
Because of this reason, we identify a branch $B$ as the function $\bigcup B$.
\emph{Cantor tree} is the tree $2^{\le \om}$.
We say that $\sigma :2^{<\om} \to 2^{<\alpha}$ is an \emph{embedding}
if $s < t \iff \sigma(s) <\sigma(t)$ for every $s,  t \in 2^{<\om}$.
Every embedding $\sigma:2^{<\om} \to 2^{<\alpha}$  canonically induces the map
$\sigma^*:2^{\om} \to 2^{\le \alpha }$ as
$\sigma^*(f)=\bigcup_{n<\om} \sigma(f \restriction n)$.
Note that a tree $T$ does not contain an isomorphic copy of Cantor tree
if and only if for every embedding $\sigma:2^{<\om} \to T$
there is $f \in 2^\om$ with $\sigma^*(f) \notin T$.

\begin{prop}\label{prop2}
Assume CH. Suppose there exists a tree $T \subseteq 2^{<\om_2}$ such that:
\begin{enumerate}
\item Each level of $T$ has cardinality at most $\om_1$.
\item $T$ has no branch of size $\om_2$.
\item $\size{T} >2^\om$, or $T$ has strictly more than $2^{\om}$ many branches.
\item $T$ does not contain an isomorphic copy of Cantor tree.
\end{enumerate}
Then there exist  zero-dimensional $T_1$ Lindel\"of spaces $X$, $X_0$, $X_1$ which
satisfy the assumptions of Proposition \ref{prop1}.
\end{prop}

Now Theorem \ref{thm1} follows from Propositions \ref{prop1} and \ref{prop2}:
If there exists an $\om_1$-Kurepa tree $T \subseteq 2^{<\om_1}$,
by CH, $T$ satisfies the assumptions of Proposition \ref{prop2}.
If $\square(\om_2)$ holds,
then there is an $\om_2$-Aronszajn tree $T \subseteq 2^{<\om_2}$ which does not contain an isomorphic copy of Cantor tree
(Todor\v cevi\'c, (1.11) in \cite{T}.
See also Corollary 3.10 in K\"onig \cite{Konig}).
It is clear that $T$ fulfills the assumptions of Proposition \ref{prop2}.

We start the proof of Proposition \ref{prop2}.
Fix a tree $T$ satisfying the assumptions.
We may assume that every $t \in T$ has two immediate successors $t^\frown \seq{0},t^\frown \seq{1}$ in $T$.


Let $T^*=\{t \in T \mid \cf(\dom(t))=\om_1\}$.
For $i=0,1,$ let $\calB_i$ be the set of all branches $B$ of $T$ with
$\cf(\dom(B))=\om_i$.
For $t \in T$, let $[t]=\{B \in \calB_0 \cup \calB_1 \mid t \in B\} \cup\{s \in T^* \mid t \le s\}$ and $[t]^{+}=[t^\frown \seq{0}] \cup [t^\frown \seq{1}]$.
Note that if $t \in T \setminus T^*$ then $[t]=[t]^+$.

First we define the space $X$.
The underlying set of $X$ is $\calB_0 \cup \calB_1 \cup T^*$.
The topology is generated by the family
\[
\{[t] \mid t \in T \setminus T^* \} \cup \{ [s]\setminus [t]^+ \mid 
t \in T^*, s \notin T^*, s<t\}\]
as an open base.
It is routine to check that $X$ is a zero-dimensional $T_1$ space of size $>2^\om$.
For $t \in T^*$, the family $\{ [t \restriction \alpha] \setminus [t]^+ \mid \alpha<\dom(t), \cf(\alpha) \neq \om_1
\}$ is a local base for $t$, and $\chi(t, X)=\om_1$.
For $B \in \calB_0 \cup \calB_1$, the family $\{[B \restriction \alpha] \mid \alpha<\dom(B),\cf(\alpha) \neq \om_1 \}$ is a local base for $B$.
It is clear that $\chi(B,X)=\om_i \iff B \in \calB_i$.

We prove that $X$ is Lindel\"of. 

\begin{claim}\label{3.2+}
$X$ is Lindel\"of.
\end{claim}
\begin{proof}
Let $\calU$ be an open cover of $X$.
Let $T_{\calU}$ be the set of all $t \in T$ such that
there is no countable subfamily $\calV \subseteq \calU$ with
$[t] \subseteq \bigcup \calV$.
If $T_{\calU} =\emptyset$, then 
$[\emptyset] \subseteq \calV$ for some countable $\calV \subseteq \calU$,
and $\calV$ is a countable cover of $X$.
Thus it is enough to see that $T_\calU =\emptyset$.

Suppose to the contrary that $T_\calU \neq \emptyset$.
We note that for $t \in T_{\calU}$ and $s \in T$, if $s \le t$ then $s \in T_{\calU}$.
Hence $T_\calU$ is a subtree of $T$.

First we check that $T_{\calU}$ has no maximal element.
Suppose not and take $t \in T_{\calU}$ which is a maximal element of $T_{\calU}$.
Then $t^\frown \seq{0}, t^\frown\seq{1}$ are elements of $T$ but not of $T_{\calU}$.
Thus there are countable subfamilies $\calV_0, \calV_1 \subseteq \calU$ with
$[t^\frown \seq{i}] \subseteq \bigcup \calV_i$ for $i=0,1$.
If $t \notin T^*$, then $[t]= [t]^+ \subseteq \bigcup (\calV_0 \cup \calV_1)$,
thus we have $t \in T_{\calU}$. This is a contradiction.
If $t \in T^*$,
pick $O \in \calU$ with $t \in O$.
Then $[t]= \{t\}\cup [t]^+ \subseteq O \cup \bigcup (\calV_0 \cup \calV_1)$,
this is a contradiction too.

Next we check that $T_{\calU}$ is branching.
Suppose not, and take $t_0 \in T_{\calU}$ 
such that every $t \in T_{\calU}$ with $t_0 \le t$ has only one immediate successor in $T_{\calU}$.
Let $C=\{t \in T_{\calU} \mid t_0 \le t\}$.
$C$ is a chain of $T$. By the assumption,
we have that $\size{C} \le \om_1$.
Let $\seq{t_\alpha \mid \alpha<\gamma}$ be the increasing enumeration of $C$.
We know that $\gamma$ is a limit ordinal with $\gamma<\om_2$.
By induction on $\alpha<\gamma$, we claim that 
there is a countable $\calV \subseteq \calU$ with $[t_0] \setminus [t_\alpha] \subseteq \bigcup \calV$.
The case $\alpha=0$ is trivial.
If $\alpha=\beta+1$ and $\cf(\beta)=\om_1$,
then $t_\beta \in T^*$ and $[t_0] \setminus [t_\alpha]=
([t_0] \setminus [t_\beta]) \cup [t_\beta^\frown \seq{1-t_\alpha(\dom(t_\beta))}] \cup \{t_\beta\}$.
Take a countable $\calV \subseteq \calU$ with
$[t_0] \setminus [t_\beta] \subseteq \bigcup \calV$.
Because $t_\alpha^\frown \seq{1-t(\dom(t_\beta))} \notin T_{\calU}$,
there is a countable $\calV' \subseteq \calU$
with $[t_\beta^\frown \seq{1-t_\alpha(\dom(t_\beta))}] \subseteq \calV'$.
Then $[t_0] \setminus [t_\alpha] \subseteq O \cup \bigcup (\calV \cup \calV') $ for some $O \in \calU$ with $t_\beta \in O$.
The case that $\alpha=\beta+1$ and $\cf(\beta) \neq \om_1$ is similar.
Suppose $\alpha$ is a limit ordinal.
If $\cf(\alpha)=\om$,
take an increasing sequence $\seq{\alpha_n \mid n<\om}$ with limit $\alpha$.
By the induction hypothesis,
for $n<\om$ there is a countable $\calV_n \subseteq \calU$
with $[t_0] \setminus [t_{\alpha_n}] \subseteq \bigcup \calV_n$.
$[t_0] \setminus [t_\alpha]=\bigcup_{n<\om} ([t_0] \setminus [t_{\alpha_n}])$,
hence $[t_0] \setminus [t_\alpha] \subseteq \bigcup_{n<\om} \calV_n$.
Finally suppose $\cf(\alpha)=\om_1$.
Then $t_\alpha \in T^*$. Pick $O \in \calU$ with $t_\alpha \in O$.
By the definition of the topology of $X$,
there is some $s<t_\alpha$ such that $s \notin T^*$ and 
$[s] \setminus [t_\alpha]^+ \subseteq O$. Fix $\beta<\alpha$ with $s \le t_\beta$,
and take a countable $\calV \subseteq \calU$ with $[t_0] \setminus [t_\beta] \subseteq \bigcup \calV$.
Then $[t_0] \setminus [t_\alpha] \subseteq ([t_0] \setminus [t_\beta]) \cup ([s]\setminus [t_\alpha]^+)
\subseteq O \cup \bigcup \calV$.

Let $t_\gamma=\bigcup_{\alpha<\gamma} t_\alpha$. We know $t_\gamma \notin T_{\calU}$.
If $t_\gamma \in T$,
by the same argument as before,
we can find a countable $\calV \subseteq \calU$ with
$[t_0] \setminus [t_\gamma] \subseteq \bigcup \calV$.
Since $t_\gamma \notin T_\calU$, there is a countable $\calV' \subseteq \calU$
such that $[t_\gamma] \subseteq \calV'$.
Then $[t_0] \subseteq \bigcup (\calV \cup \calV')$, this is a contradiction.
If $t_\gamma \notin T$, then $t_\gamma \in \calB_0 \cup \calB_1$.
Pick $O \in \calU$ with $t_\gamma \in O$.
Then there is $t \in T \setminus T^*$ with $t<t_\gamma$ and $[t] \subseteq O$.
Fix $\beta<\gamma$ with $t \le t_\beta$.
We have $[t_0] = ([t_0] \setminus [t_\beta]) \cup [t]$,
and we can derive a contradiction as before.

Now we know that $T_\calU$ has no maximal element and is branching.
Hence we can take an embedding $\sigma:2^{<\om} \to T_{\calU}$.
By the assumption on $T$, there is some $f \in 2^\om$ with
$\sigma^*(f) \notin T$.
Then $B= \sigma^*(f)$ is a branch of $T$ and $B \in \calB_0$.
Fix an open set $O \in \calU$ with $B \in O$.
There is some $t \in B$ with $[t] \subseteq O$, and
we can choose $n<\om$ with $t <\sigma(f \restriction n)$.
However then $[\sigma(f \restriction n)] \subseteq O$,
this contradicts to $\sigma(f \restriction n) \in T_{\calU}$.
\qedhere[Claim]
\end{proof}
\begin{remark}
The place where we use the assumption that ``Cantor tree $2^{\le \om}$ cannot be embedded into $T$''
is the proof of this claim, 
and the referee pointed out us that, for proving this claim, the Cantor tree assumption can be weakened to
that ``the tree $2^{<\om_1}$ cannot be embedded into $T$''.
\end{remark}

Next, by modifying open neighborhoods of points in $\calB_0$, 
we construct finer spaces $X_0$ and $X_1$.
Let us say that an embedding $\sigma$ is \emph{good} if
$\dom(\sigma^*(f))=\dom(\sigma^*(g))$ for every $f, g\in 2^\om$.
\begin{claim}\label{2.3}
For every embedding $\sigma$, there is a good embedding $\tau$
such that $\mathrm{Range}(\tau) \subseteq \mathrm{Range}(\sigma)$.
\end{claim}
\begin{proof}
First note that
the set $D=\{\dom(\sigma^*(f)) \mid f \in 2^\om\}$ is at most countable,
because $D$ is a subset of all limit points of the countable set $\{\dom(\sigma(t)) \mid t \in 2^{<\om}\}$.

Now we  have $2^\om=\bigcup_{\alpha \in D} \{f \in 2^\om \mid \dom(\sigma^*(f))=\alpha\}$.
$D$ is countable, thus there is some $\alpha \in D$
such that $E=\{f \in 2^\om \mid \dom(\sigma^*(f))=\alpha\}$ is uncountable.
It is clear that $\alpha$ is a limit ordinal with countable cofinality.
Take an increasing sequence $\seq{\alpha_i \mid i<\om}$ with limit $\alpha$.
Then for every $t \in 2^{<\om}$ and $i<\om$,
if $\{f \in E \mid t \subseteq f\}$ is uncountable,
then there are two $s_0, s_1 \in 2^{<\om}$
such that $t < s_0, s_1$, $\dom(\sigma(s_0)), \dom(\sigma(s_1)) \ge \alpha_i$, and both
$\{f \in E \mid s_0 \subseteq f\}$, $\{f \in E \mid s_1 \subseteq f\}$ are uncountable;
Since the Cantor space $2^\om$ is compact,
we can find two $f_0, f_1 \in 2^\om$ such that
$f_0, f_1 \supseteq t$, and
for every open neighborhood $O$  of $f_0$ or $f_1$ in $2^\om$,
the set $O \cap E$ is uncountable.
Take a large $n<\om$ with $\dom(\sigma(f_0 \restriction n)), \dom(\sigma(f_1 \restriction n)) \ge \alpha_i$
and $f_0 \restriction n \neq f_1 \restriction n$.
Let $s_0 =f_0 \restriction n$ and $s_1=f_1 \restriction n$.
Then we have that the sets 
$\{f \in E \mid s_0 \subseteq f\}$ and $\{f \in E \mid s_1 \subseteq f\}$ are uncountable.

Using the above observation,
we can take an embedding $\rho:2^{<\om} \to 2^{<\om}$
such that for every $t \in 2^{<\om}$ with $\dom(t)=n$,
we have $\alpha_n \le \dom(\sigma(\rho(t))) <\alpha$.
Let $\tau=\sigma \circ \rho$. 
It is easy to check that $\tau:2^{<\om} \to T$ is a required embedding.
\qedhere[Claim]
\end{proof}

\begin{claim}\label{2.4}
Let $\sigma:2^{<\om} \to T$ be a good embedding.
Then the set $\{f \in 2^\om \mid \sigma^*(f) \notin T\}$
is uncountable
\end{claim}
\begin{proof}
If it is countable, we can take an enumeration $\seq{f_n \mid n<\om}$ of it.
Then we can take an embedding $\tau :2^{<\om} \to 2^{<\om}$
such that $\sigma(\tau(t)) \neq \sigma(f_{\dom(t)} \restriction \dom(\tau(t)))$.
Let $\rho=\sigma \circ \tau$.
$\rho$ is an embedding, $\mathrm{Range}(\rho) \subseteq \mathrm{Range}(\sigma)$, and
$\mathrm{Range}(\rho^*) \cap \{\sigma^*(f) \mid f \in 2^\om, \sigma^*(f) \notin T\}=\emptyset$.
Because $T$ does not contain an isomorphic copy of Cantor tree,
there is some $f \in 2^\om$ such that $\rho^* (f) \notin T$.
$\mathrm{Range}(\rho) \subseteq \mathrm{Range}(\sigma)$, hence
$\mathrm{Range}(\rho^*) \subseteq \mathrm{Range}(\sigma^*)$ and
there is $n$ with $\rho^*(f)=\sigma^*(f_n)$,
this is a contradiction.
\qedhere[Claim]
\end{proof}

Let $G$ be the set of all good embeddings.
\begin{claim}
There is an injection $\varphi$ from $G$ into $\calB_0$
such that $\varphi(\sigma) \in \mathrm{Range}(\sigma^*)$ for every $\sigma \in G$.
\end{claim}
\begin{proof}
For $\sigma \in G$,
let $\alpha_\sigma$ be the ordinal such that
$\dom(\sigma^*(f))=\alpha_\sigma$ for every $f \in 2^\om$.
$\alpha$ is a limit ordinal with countable cofinality.

Fix a limit ordinal $\alpha$ with countable cofinality.
We define $\varphi \restriction \{\sigma \in G \mid \alpha_\sigma=\alpha\}$.
We have that $\mathrm{Range}(\sigma) \subseteq T \cap 2^{<\alpha}$ 
for every $\sigma \in G$ with $\alpha_\sigma=\alpha$.
By the assumption on $T$,
we have that $T \cap 2^{<\alpha}$ has cardinality at most $\om_1$, so
there are at most $(\om_1)^\om=\om_1$ many good embeddings $\sigma$ with $\alpha_\sigma=\alpha$.
In addition,
by Claim \ref{2.4},
for every $\sigma \in G$ with $\alpha_\sigma=\alpha$,
the set $\{f \in 2^\om \mid  \sigma^*(f) \notin T\}$ is uncountable, hence has cardinality $\om_1$.
Combining these observations, we can easily take an injection
$\varphi \restriction \{\sigma \in G \mid \alpha_\sigma=\alpha\}$ into $\calB_0$
with $\varphi(\sigma) \in \mathrm{Range}(\sigma^*)$.
\qedhere[Claim]
\end{proof}

Fix an injection $\varphi:G \to \calB_0$ with
$\varphi(\sigma) \in \mathrm{Range}(\sigma^*)$.
For $B \in \calB_0$, let $\delta_B=\dom(B)$.
We define an increasing sequence $\seq{\delta_n^B \mid n<\om}$ with
limit $\delta_B$ as follows.
If $B \notin \mathrm{Range}(\varphi)$, then 
$\seq{\delta_n^B \mid n<\om}$ is an arbitrary increasing sequence with limit $\delta_B$
and $\cf(\delta_n^B) \neq \om_1$.
If $B \in \mathrm{Range}(\varphi)$, there is a unique $\sigma \in G$
with $\varphi(\sigma)=B$.
Take $f \in 2^\om$ with
$\sigma^*(f) =B$.
Then take an increasing sequence $\seq{\delta_n^B \mid n<\om}$
with limit $\delta_B$ such that 
$\cf(\delta_n^B) \neq \om_1$ and
for each $n<\om$
there is $m<\om$ with $B \restriction  \delta_n^B<s
<
B \restriction  \delta_{n+1}^B$,
where $s$ is a maximal element of $T$ with
$s <\sigma(f \restriction m+1), \sigma(f \restriction m^\frown \seq{1-f(m)})$.

Now we are ready to define $X_0$ and $X_1$.
For $B \in \calB_0$ and $m<\om$,
let $W_0(B,m)= \{B\} \cup \bigcup\{ [B \restriction \delta^B_n] \setminus [B \restriction \delta^B_{n+1}]
\mid n: \text{even}, n>m\}$
and $W_1(B,m)= \{B\} \cup \bigcup \{ [B \restriction \delta^B_n] \setminus [B \restriction \delta^B_{n+1}]
\mid n: \text{odd}, n>m\}$.
The topology of $X_0$ is generated
by the family 
\[
\{[t] \mid t \in T \setminus T^* \}  \cup \{ [s]\setminus [t]^+ \mid 
t \in T^*, s \notin T^*, s<t\}
\cup \{ W_0(B, m) \mid B \in \calB_0, m <\om\}
\]
as an open base.
The topology of $X_1$ is generated
by the family 
\[
\{[t] \mid t \in T \setminus T^* \}  \cup \{ [s]\setminus [t]^+ \mid 
t \in T^*, s \notin T^*, s<t\}
\cup \{ W_1(B, m) \mid B \in \calB_0, m <\om\}
\]
as an open base.
It is not hard to check that $X_0$ and $X_1$ are zero-dimensional $T_1$ spaces finer than $X$.
We have to check that $X_0$ and $X_1$ satisfy the assumptions in Proposition \ref{prop1}.

For $B \in \calB_0$, the family 
$\{ W_0(B, m) \mid m<\om\}$ forms a local base for $B$ in $X_0$,
and
$\{ W_1(B, m) \mid m<\om\}$ forms a local base for $B$ in $X_1$.
Moreover
$W_0(B,0)  \cap W_1(B, 0)=\{B\}$.

For $B \in \calB_1$, take an increasing continuous sequence $\seq{\delta_\alpha \mid \alpha<\om_1}$
with limit $\dom(B)$ and $\cf(\delta_\alpha) \neq \om_1$.
Then $\{[B \restriction \delta_\alpha] \mid \alpha<\om_1\}$ is a continuously decreasing sequence
of clopen sets in $X$ with $\bigcap_{\alpha<\om_1} [B \restriction \delta_\alpha]=\{B\}$.
Similarly, for $t \in T^*$,
take an increasing continuous sequence $\seq{\delta_\alpha \mid \alpha<\om_1}$
with limit $\dom(t)$ and $\cf(\delta_\alpha) \neq \om_1$.
Then
the sequence $\{[t \restriction \delta_\alpha] \setminus [t]^+\mid \alpha<\om_1\}$ is a required one.

Finally we have to check that $X_0$ and $X_1$ are Lindel\"of.
\begin{claim}\label{3.7}
$X_0$ and $X_1$ are Lindel\"of.
\end{claim}
\begin{proof}
We only show that $X_0$ is Lindel\"of.
One can check that $X_1$ is also Lindel\"of by the same way.

Let $\calU$ be an open cover of $X_0$.
As before, let $T_\calU$ be the set of all $t \in T$
such that there is no countable $\calV \subseteq \calU$ with
$[t] \subseteq \calV$.
It is enough to see that $T_{\calU}=\emptyset$.
Suppose to the contrary that $T_{\calU} \neq \emptyset$.
We can see that $T_{\calU}$ has no maximal element.
Next we check that $T_{\calU}$ is branching.
If not, then we can take a chain $\seq{t_\alpha\mid \alpha<\gamma}$ in $T_{\calU}$.
By the same argument as before, we know that
for every $\alpha<\gamma$ there is a countable $\calV \subseteq \calU$ with
$[t_0] \setminus [t_\alpha] \subseteq \calV$.
Let $t_\gamma =\bigcup_{\alpha<\gamma} t_\alpha$.
If $t_\gamma \in \calB_1$ or $t_\gamma \in T$,
then one can derive a contradiction as before.
If $t_\gamma \in \calB_0$,
take an increasing sequence $\seq{\alpha_n \mid n<\om}$ with limit $\gamma$.
For $n<\om$, take a countable $\calV_n \subseteq \calU$
with $[t_0]\setminus [t_{\alpha_n}] \subseteq \calV_n$.
Pick an open set $O \in \calU$ with $t_\gamma \in O$.
Then $[t_0] 
=\bigcup_{n<\om} ([t_0] \setminus [t_{\alpha_n}]) \cup \{t_\gamma\}
\subseteq O \cup \bigcup_{n<\om} \calV_n$, this is a contradiction.

Now we have that $T_\calU$ has no maximal element and is branching.
Hence there is an embedding $\sigma:2^{<\om} \to T_{\calU}$.
By Claim \ref{2.3}, there is a good embedding $\tau$ with
$\mathrm{Range}(\tau) \subseteq \mathrm{Range}(\sigma)$.
Consider $B=\varphi(\tau) \in \calB_0$.
Take $f \in 2^\om$ with $\tau^*(f)=B$.
Fix an open set $O \in \calU$ with $B \in O$.
Then there is $m<\om$ such that
$W_0(B,m) \subseteq O$,
so there is an odd  number $n^*$ 
with $[B \restriction \delta^B_{n^*}] \setminus [B \restriction \delta^B_{n^*+1}] \subseteq O$.
By the choice of $\delta^B_{n^*}$,
there is some $l<\om$ with
$B \restriction  \delta_{n^*}^B<s
<
B \restriction  \delta_{n^*+1}^B$,
where $s$ is a maximal element of $T$ with
$s <\tau(f \restriction l+1), \tau(f \restriction l^\frown \seq{1-f(l)})$.
This means that
$[\tau(f \restriction l^\frown \seq{1-f(l)})]
\subseteq 
[B \restriction \delta^B_{n^*}] \setminus [B \restriction \delta^B_{n^*+1}] $,
hence
$[\tau(f \restriction l^\frown \seq{1-f(l)})] \subseteq O$.
This contradicts to $\tau(f \restriction l^\frown \seq{1-f(l)}) \in T_\calU$.
\qedhere[Claim]
\end{proof}
\begin{remark}
As in the proof of Claim \ref{3.2+}, we used the assumption that ``Cantor tree $2^{\le \om}$ cannot be embedded into $T$''
in the proof of this claim.
However, unlike Claim \ref{3.2+},
the author does not know whether it can be weakened to
that ``the tree $2^{<\om_1}$ cannot be embedded into $T$''.
\end{remark}

\end{document}